\documentclass[11pt,a4paper]{amsart}
\usepackage{enumerate,amsmath,amsthm,amssymb,url}
\usepackage{amsfonts,amsmath,graphicx,float,graphicx,url,calc}
\usepackage[numbers,sort&compress]{natbib}
\usepackage[lmargin=35mm,rmargin=35mm,tmargin=34mm,bmargin=34mm]{geometry}

\newcommand{\half}{\ensuremath{\protect\tfrac{1}{2}}}

\theoremstyle{plain}
\newtheorem{theorem}{Theorem}
\newtheorem{lemma}[theorem]{Lemma}
\newtheorem{corollary}[theorem]{Corollary}

\newtheorem{claim}[theorem]{Claim}
\theoremstyle{definition}
\newtheorem{conjecture}[theorem]{Conjecture}

\newcommand{\spacing}[1]{
\renewcommand{\baselinestretch}{#1}
\setlength{\footnotesep}{\baselinestretch\footnotesep}}
\spacing{1.2}

\newcommand{\thmlabel}[1]{\label{thm:#1}}
\newcommand{\thmref}[1]{Theorem~\ref{thm:#1}}
\newcommand{\twothmref}[2]{Theorems~\ref{thm:#1} and \ref{thm:#2}}
\newcommand{\lemlabel}[1]{\label{lem:#1}}
\newcommand{\lemref}[1]{Lemma~\ref{lem:#1}}
\newcommand{\eqnlabel}[1]{\label{eqn:#1}}
\newcommand{\eqnref}[1]{\eqref{eqn:#1}}
\newcommand{\twoeqnref}[2]{\eqref{eqn:#1} and \eqref{eqn:#2}}
\newcommand{\figlabel}[1]{\label{fig:#1}}
\newcommand{\figref}[1]{Figure~\ref{fig:#1}}
\newcommand{\figsref}[1]{Figures~\ref{fig:#1}}
\newcommand{\seclabel}[1]{\label{sec:#1}}
\newcommand{\secref}[1]{Section~\ref{sec:#1}}
\newcommand{\conjlabel}[1]{\label{con:#1}}
\newcommand{\conjref}[1]{Conjecture~\ref{con:#1}}    
\newcommand{\corlabel}[1]{\label{cor:#1}}
\newcommand{\corref}[1]{Corollary~\ref{cor:#1}}
\newcommand{\claimlabel}[1]{\label{claim:#1}}
\newcommand{\claimref}[1]{Claim~\ref{claim:#1}}

\newcommand{\Figure}[4][htb]{
\begin{figure}[#1]
	\vspace*{1ex}
	\begin{center}#3\end{center}
	\vspace*{-1ex}
	\caption{\figlabel{#2}#4}
\end{figure}
}

\DeclareMathOperator{\conv}{conv}
\DeclareMathOperator{\ES}{ES}
\DeclareMathOperator{\dist}{dist}

\newcommand{\FLOOR}[1]{\ensuremath{\protect\left\lfloor#1\right\rfloor}}

\begin{document}

\title[Empty Pentagons]{Every Large Point Set contains\\ Many Collinear Points or an Empty Pentagon}


\author [Abel et al.]{Zachary Abel}
\address{ Department of Mathematics, 
 Harvard University
\newline Cambridge, Massachusetts, U.S.A.}
\email{zabel@fas.harvard.edu}

\author[]{Brad Ballinger}
\address{ Davis School for Independent Study
\newline Davis, California, U.S.A.}
\email{ballingerbrad@yahoo.com}

\author[]{Prosenjit Bose}
\address{ School of Computer Science, 
 Carleton University
\newline Ottawa, Canada}
\email{jit@scs.carleton.ca}

\author[]{S\'ebastien Collette}
\address{ D\'epartement d'Informatique, 
 Universit\'e Libre de Bruxelles
\newline Brussels, Belgium}
\email{sebastien.collette@ulb.ac.be}

\author[]{Vida Dujmovi\'{c}}
\address{ School of Computer Science, 
 Carleton University 
\newline Ottawa, Canada}
\email{vida@scs.carleton.ca}

\author[]{Ferran Hurtado}
\address{Departament de Matem{\`a}tica Aplicada II, 
 Universitat Polit{\`e}cnica de Catalunya
\newline  Barcelona, Spain}
\email{ferran.hurtado@upc.edu}
 
\author[]{Scott D. Kominers}
\address{ Department of Mathematics, 
 Harvard University
\newline Cambridge, Massachusetts, U.S.A.}
\email{kominers@fas.harvard.edu;skominers@gmail.com}

\author[]{Stefan Langerman}
\address{Ma\^itre de Recherches du F.R.S.-FNRS
\newline D\'epartement d'Informatique, 
 Universit\'e Libre de Bruxelles
\newline Brussels, Belgium}
\email{stefan.langerman@ulb.ac.be}

\author[]{Attila P\'or}
\address{ Department of Mathematics, 
  Western Kentucky University
\newline Bowling Green,  Kentucky, U.S.A.}
\email{attila.por@wku.edu}

\author[]{David~R.~Wood}
\address{ Department of Mathematics and Statistics, 
 The University of Melbourne
\newline Melbourne, Australia}
\email{woodd@unimelb.edu.au}

\begin{abstract}
We prove the following generalised empty pentagon theorem: for every integer $\ell \geq 2$, every sufficiently large set of points in the plane contains $\ell$ collinear points or an empty pentagon.  As an application, we settle the next open case of the ``big line or big clique'' conjecture of K\'ara, P\'or, and Wood [\emph{Discrete Comput.\ Geom.}\ 34(3):497--506, 2005].  
\end{abstract}



\maketitle

\section{Introduction}

While the majority of theorems and problems about sets of points in the plane assume that the points are in general position,
there are many interesting theorems and problems about sets of points with collinearities. 
The Sylvester-Gallai Theorem and the orchard problem are some examples; see \citep{BMP}. 
The main contribution of this paper is to extend the `empty pentagon' theorem about point sets in general position to point sets with collinearities.

\subsection{Definitions}

We begin with some standard definitions. 
Let $P$ be a finite set of points in the plane. We say that $P$ is in \emph{general position} if no three points in $P$ are collinear.
Let $\conv(P)$ denote the convex hull of $P$. 
We say that $P$ is in \emph{convex position} if every point of $P$ is on the boundary of $\conv(P)$.
A point $v\in P$ is a \emph{corner} of $P$ if $\conv(P-v)\neq\conv(P)$.
We say that $P$ is in \emph{strictly convex position} if each point of $P$ is a corner of $P$.
A \emph{strictly convex $k$-gon} is the convex hull of $k$ points in strictly convex position.
If $X\subseteq P$ is a set of $k$ points in strictly convex position  and $\conv(X)\cap P=X$, 
then $\conv(X)$ is called a \emph{$k$-hole} (or an \emph{empty strictly convex $k$-gon}) of $P$.
A $4$-hole is called an \emph{empty quadrilateral}, 
a $5$-hole is called an \emph{empty pentagon}, 
a $6$-hole is called an \emph{empty hexagon}, etc.

For distinct points $a,b,c$ in the plane, let $\Delta[a,b,c]$ be the closed triangle determined by $a,b,c$, 
and let $\Delta(a,b,c)$ be the open triangle determined by $a,b,c$.
For integers $n\leq m$, let $[n,m]:=\{n,n+1,\dots,m\}$ and $[n]:=[1,n]$.

\subsection{Erd\H{os}-Szekeres Theorem}

The Erd\H{os}-Szekeres Theorem \citep{ES35} states that for every integer $k$ there is a minimum integer $\ES(k)$ such that every set of at least $\ES(k)$ points in general position in the plane contains $k$ points in convex position (which are therefore in strictly convex position). See \citep{MorrisSoltan-BAMS00,BK-LNCS01,BMP,TothValtr05} for surveys of this theorem. The following generalisation of the Erd\H{os}-Szekeres Theorem for point sets with collinearities is easily proved by applying a suitable perturbation of the points (see \secref{ErdosSzekeres}):

\begin{theorem}
\thmlabel{ErdosSzekeres} For every integer $k$ every set of at least $\ES(k)$ points in the plane contains $k$ points in  convex position.
\end{theorem}

The Erd\H{os}-Szekeres Theorem generalises for points in strictly convex position as follows:

\begin{theorem} 
\thmlabel{StrictErdosSzekeres}
For all integers $\ell\geq2$ and $k\geq3$ there is a minimum integer $\ES(k,\ell)$ such that  every set of at least  $\ES(k,\ell)$ points in the plane contains:
\begin{itemize}
\item $\ell$ collinear points, or
\item $k$ points in strictly  convex position.
\end{itemize} 
\end{theorem}

Of course, the conclusion in \thmref{StrictErdosSzekeres} that there is a large set of collinear points is unavoidable, since  a large set of collinear points only contains two points in strictly convex position. \thmref{StrictErdosSzekeres} is known (it is Exercise 3.1.3 in \citep{Mat02}), but as far as we are aware, no proof of it has appeared in the literature and no explicit bounds on $\ES(k,\ell)$ have been formulated. To illustrate various proof techniques in geometric Ramsey theory, we present three proofs of \thmref{StrictErdosSzekeres}. The first proof finds a large subset of points in general position and then applies the standard Erd\H{o}s-Szekeres theorem (see \lemref{ErdosSzekeresGenPos}). The second proof first applies \thmref{ErdosSzekeres} to obtain a large subset in convex position, in which a large subset in strictly convex position is found (see \lemref{ErdosSzekeresConvexToStrict}). The third proof is based on Ramsey's Theorem for hypergraphs (see \secref{Quads}).


\subsection{Empty Polygons}

Attempting to strengthen the Erd\H{os}-Szekeres Theorem, \citet{Erdos75} asked whether for each fixed $k$ every sufficiently large set of points in general position contains a $k$-hole. \citet{Harborth78} answered this question in the affirmative for $k\leq 5$, by showing that every set of at least ten points in general position contains a $5$-hole. On the other hand, \citet{Horton83} answered Erd\H{o}s' question in the negative for $k\geq7$, by constructing arbitrarily large sets of points in general position that contain no $7$-hole. The remaining case of $k=6$ was recently solved independently by  \citet{Gerken-DCG08} and \citet{Nicolas-DCG07}, who proved that every sufficiently large set of points in general position contains a $6$-hole.
See \citep{WuDing-AML08,WuDing07,Valtr-EmptyHexagon,Koshelev-DAN07,BaranyValtr-SSMH04,Dumitrescu-SSMH00,VLK-PMH07,Valtr-DCG02,Valtr-SSMH95,Valtr-DCG92,PRS-JCTA06,Hosono-DM05,DuDing-JAMC05,BHKU-PMH04,Nyklova-SSMH03,KL-PMH03,Overmars-DCG03,DEO-Algo90} for more on empty convex polygons.

The above results do not immediately generalise to sets with a bounded number of collinear points by simply choosing a large subset in general position as in the first proof of the Erd\H{os}-Szekeres Theorem (since the deleted points might `fill a hole'). Nevertheless, we prove the following `generalised empty pentagon' theorem, which is the main contribution of this paper (proved in \secref{Pentagons}).

\begin{theorem}
\thmlabel{GenFiveHole}
For every integer $\ell\geq2$, every finite set of at least $\ES(\frac{(2\ell-1)^{\ell}-1}{2\ell-2} )$ points in the plane contains
\begin{itemize}
\item $\ell$ collinear points, or
\item a $5$-hole.
\end{itemize}
\end{theorem}

Note that \citet{Eppstein-SOCG07} characterised the point sets with no $5$-hole in terms of the acyclicity of an associated \emph{quadrilateral graph}. 
However, it is not clear how this result can be used to prove \thmref{GenFiveHole}. Earlier, \citet{Rabinowitz} defined a set of points with no $5$-hole to have the \emph{pentagon property}.

\subsection{Big Line or Big Clique Conjecture}

\thmref{GenFiveHole} has an important ramification for the following ``big line or big clique'' conjecture by \citet{KPW-DCG05}.
Let $P$ be a finite set of points in the plane. Two distinct points $v,w\in P$ are \emph{visible} with respect to $P$ if $P\cap\overline{vw}=\{v,w\}$, where $\overline{vw}$ denotes the closed line segment between $v$ and $w$. The \emph{visibility graph} of $P$  has vertex set $P$, where two distinct points $v,w\in P$ are adjacent if and only if they are visible with respect to $P$. 

\begin{conjecture}[\citep{KPW-DCG05}]
\conjlabel{BigClique}
For all integers $k$ and $\ell$ there is an integer $n$ such that every finite set of at least $n$ points in the plane contains:
\begin{itemize}
\item $\ell$ collinear points, or 
\item $k$ pairwise visible points (that is, the visibility graph contains a $k$-clique).
\end{itemize}
\end{conjecture}


\conjref{BigClique} has  recently attracted considerable attention \citep{Matousek08,Luigi,KPW-DCG05}. 
It is trivially true for $\ell\leq3$ and all $k$.
\citet{KPW-DCG05} proved it for $k\leq 4$ and all $\ell$.
\citet{Luigi} proved it in the case that $k=5$ and $\ell=4$.  
Here we prove the next case of \conjref{BigClique} for infinitely many values of $\ell$. 

\begin{theorem}
\thmlabel{BigCliqueCase}
\conjref{BigClique} is true for $k=5$ and all $\ell$.
\end{theorem}

\begin{proof}
By \thmref{GenFiveHole}, every sufficiently large set of points contains $\ell$ collinear points (in which case we are done) or a $5$-hole $H$. 
Let $H'$ be a $5$-hole contained in $H$ with minimum area. Then the corners of $H'$ are five pairwise visible points (otherwise there is a $5$-hole contained in $H$ with less area, as illustrated in \figref{FiveHole}).
\end{proof}

\Figure{FiveHole}{\includegraphics{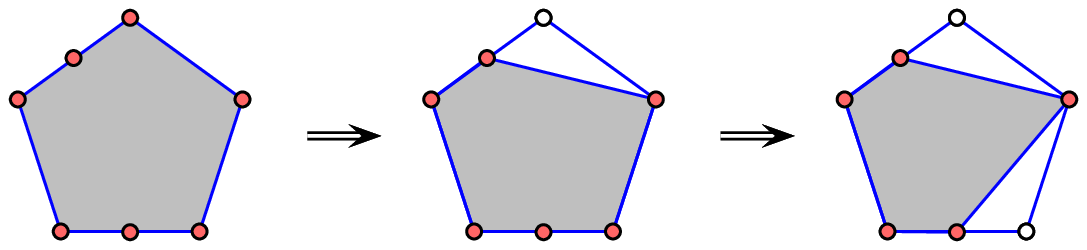}}{Every 5-hole contains five pairwise visible points.}

\section{Points in Convex Position}
\seclabel{Convex}

In this section we consider the following problem (which will be relevant to the proofs of \lemref{ErdosSzekeresConvexToStrict} and \thmref{GenFiveHole} to come): given a set $P$ of points in convex position, choose a large subset of $P$ in strictly convex position. For integers $k\geq1$ and $\ell\geq1$, let $q(k,\ell)$ be the minimum integer such that every set of at least $q(k,\ell)$ points in the plane in convex position contains $\ell$ collinear points or $k$ points in strictly convex position.  Trivially, if $k\leq2$ or $\ell\leq2$ then $q(k,\ell)=\min\{k,\ell\}$. 
Since every set of points with no three points in strictly convex position is collinear, $q(3,\ell)=\ell$ for all $\ell\geq1$. 
Since every set of points in convex position with no three collinear points is in strictly convex position, $q(k,3)=k$ for all $k\geq1$. 

\begin{lemma}
\lemlabel{ConvexToStrict}
For all $\ell\geq3$ and $k\geq3$,
\begin{equation}
\eqnlabel{ConvexToStrict}
q(k,\ell)=\begin{cases}
  \half(\ell-1)(k-1)+1& \text{, if $k$ is odd}\\
  \half(\ell-1)(k-2)+2& \text{, if $k$ is even.}
\end{cases}
\end{equation}
\end{lemma}

\begin{proof}
Let $f(k,\ell)$ denote the right-hand-side of \eqnref{ConvexToStrict}.

We first prove the lower bound on $q(k,\ell)$ for odd $k\geq5$, the case $k=3$ having been proved above. 
As illustrated in \figref{Even}(a), 
let $P$ be a set consisting of $\ell-1$ points on every second side of a convex $(k-1)$-gon.
Thus $P$ has $\half(k-1)(\ell-1)$ points with no $\ell$ collinear points and no $k$ in strictly convex position 
(since at most two points from each side are in strictly convex position).
Hence $q(k,\ell)>\half(\ell-1)(k-1)$, which is an integer. Thus $q(k,\ell)\geq\half(\ell-1)(k-1)+1=f(k,\ell)$.

Now we prove the lower bound on $q(k,\ell)$ for even $k\geq4$. 
For $k=4$, a set of $\ell-1$ collinear points plus one point off the line has no four points in strictly convex position; hence $q(4,\ell)\geq\ell+1$. 
Now assume $k\geq6$. As illustrated in \figref{Even}(b), 
let $P$ be a set consisting of $\ell-1$ points on every second side of a convex $(k-2)$-gon, plus one more point not collinear with any two other points.
Thus $P$ has $\half(\ell-1)(k-2)+1$ points with no $\ell$ collinear points and no $k$ in strictly convex position 
(since at most two points from each `long' side are in strictly convex position plus the one extra point).
Hence $q(k,\ell)>\half(\ell-1)(k-2)+1$, which is an integer. Thus $q(k,\ell)\geq\half(\ell-1)(k-2)+2=f(k,\ell)$.

\Figure{Even}{\includegraphics{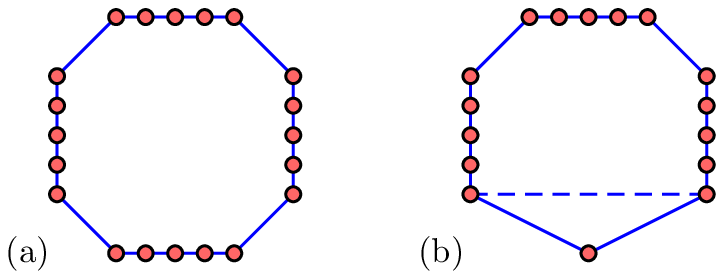}}{Extremal examples for $\ell=6$ and (a) $k=9$ and (b) $k=8$.}

We now prove the upper bound $q(k,\ell)\leq f(k,\ell)$ for $\ell\geq3$ and $k\geq1$.
We proceed by induction on $k\geq1$.  
The cases $k\in\{1,2,3\}$ or $\ell=3$ follow from the discussion at the start of the section. Now assume that $k\geq4$ and $\ell\geq4$.
Let $P$ be a set of at least $f(k,\ell)$ points in convex position with no $\ell$ collinear points and no $k$ points in strictly convex position.
Let $v_1,\dots,v_m$ be the corners of $P$ in clockwise order, where $v_{m+1}:=v_1$ and $v_0:=v_m$.
Let $P_i:=P\cap\overline{v_iv_{i+1}}$ for each $i\in[m]$.
Thus $|P_i|\in[2,\ell-1]$ for each $i\in[m]$.

Suppose that $|P_i|\geq4$ for some $i\in[m]$. Thus
$|P-P_i|\geq f(k,\ell)-(\ell-1)=f(k-2,\ell)$.
By induction, $P-P_i$ has a  subset $S$ of $k-2$ points in strictly convex position (since $P$ and thus $P-P_i$ has no $\ell$ collinear points). 
Thus $S$ plus two internal points on $P_i$ form a  subset of $k$ points in strictly convex position, which is a contradiction.
Now assume that $|P_i|\leq3$ for all $i\in[m]$.

Suppose that $|P_i|=2$ for some $i\in[m]$.
Say $t,u,v,w,x,y$ are the consecutive points on the boundary of $\conv(P)$, where $P_i=\{v,w\}$.
Since $\{u,v,w,x\}$ are in strictly convex position, assume that $k\geq5$.
Thus $|P-\{t,u,v,w,x,y\}|\geq f(k,\ell)-6\geq f(k-4,\ell)$.
By induction, $P-\{t,u,v,w,x,y\}$ has a  subset $S$ of $k-4$ points in strictly convex position (since $P$ and thus $P-\{t,u,v,w,x,y\} $ has no $\ell$ collinear points). 
Since $|P_{i-1}|\leq3$ and $|P_{i+1}|\leq3$, it follows that $S\cup\{u,v,w,x\}$ is a set of $k$ points in strictly convex position, which is a contradiction.

Now assume that $|P_i|=3$ for all $i\in[m]$.
Thus $|P|=2m$. 
As illustrated in \figref{Pi3}, let $S$ consist of each of the $m$ non-corner points of $P$, plus every second corner point 
(where in the case of odd $m$, we omit two consecutive corners from $S$). 
Thus $S$ is a set of at least $\frac{1}{2}(3m-1)$ points in strictly convex position. 
We have $|P| \ge f(k,\ell)$, which, since $\ell \geq 4$, is at least $\frac{3}{2}k-1$.
Since no $k$ points are in strictly convex position, $|S| \leq k-1$ and
$$8(k-1)\geq 8|S|\geq12m-4=6|P|-4\geq 6(\tfrac{3}{2}k-1)-4=9k-10\enspace,$$ 
implying $k\leq2$, which is a contradiction.
\end{proof}

\Figure{Pi3}{\includegraphics{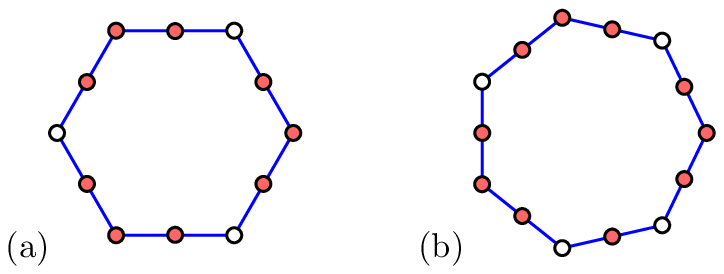}}{The case $|P_i|=3$ for all $i\in[m]$ where (a) $m=6$ and (b) $m=7$. Dark points are in $S$.}

\section{Generalisations of the Erd\H{o}s-Szekeres Theorem}
\seclabel{ErdosSzekeres}

In this section we prove  \twothmref{ErdosSzekeres}{StrictErdosSzekeres}, which generalise the Erd\H{o}s-Szekeres Theorem for points in general position.
If $P'$ is a perturbation of a finite set $P$ of points in the plane, 
let $v'\in P'$ denote the image of a point $v\in P$, and let $S':=\{v':v\in S\}$ for each $S\subseteq P$.
If $\dist(v,v')\leq\epsilon$ for each $v\in P$ then $P'$ is an \emph{$\epsilon$-perturbation}.
Observe that \thmref{ErdosSzekeres} follows from the next lemma and the Erd\H{o}s-Szekeres Theorem for points in general position (applied to $P'$).

\begin{lemma}
\lemlabel{Perturbation}
For every finite set $P$ of points in the plane, 
there is a general position perturbation $P'$ of $P$, 
such that if $S'$ is a subset of $P'$ in convex position, then
$S$ is in (non-strict) convex position.
\end{lemma}

\begin{proof}
For each non-collinear ordered triple $(u,v,w)$ of points in $P$ there exists $\mu>0$ such that every $\epsilon$-perturbation of $P$ will not change the orientation\footnote{The \emph{orientation} of an  ordered triple of points $(u,v,w)$ is $0$ if $u,v,w$ are collinear; otherwise it is \emph{positive} or \emph{negative} depending on whether we turn left or right when going from $u$ to $w$ via $v$.} of $(u,v,w)$ whenever $0<\epsilon<\mu$. 
Since there are finitely many such triples there is a minimal such $\mu$.
Let $P'$ be a $\mu$-perturbation of $P$ in general position.

Let $S'$ be a subset of $P'$ in convex position.
Consider $S'$ in  anticlockwise order.  
Thus each ordered triple of consecutive points in $S'$ has positive orientation.
Now consider $S$ in the corresponding order as $S'$.
Since the perturbation preserved negatively oriented triples, 
each ordered triple of consecutive points in $S$ has non-negative orientation.
That is, $S$ is in (non-strict) convex position, as desired.
\end{proof}




We now prove two lemmas, each of which shows how to force $k$ points in strictly convex position, thus proving \thmref{StrictErdosSzekeres}.

\begin{lemma}
\lemlabel{ErdosSzekeresConvexToStrict} 
For all $k\geq3$ and $\ell\geq3$, if $k$ is odd then
$$\ES(k,\ell)\leq\ES(\half(k-1)(\ell-1)+1)\enspace,$$
and if $k$ is even then
$$\ES(k,\ell)\leq\ES(\half(k-2)(\ell-1)+2)\enspace.$$
\end{lemma}

\begin{proof}
For odd $k$, let $P$ be a set of at least $\ES(\half(k-1)(\ell-1)+1)$ points with no $\ell$ points collinear.
Thus $P$ contains $\half(k-1)(\ell-1)+1$ points in convex position  by \thmref{ErdosSzekeres}.
Thus $P$ contains $k$ points in strictly convex position by \lemref{ConvexToStrict}.
The proof for even $k$ is analogous.
\end{proof}

\begin{lemma}
\lemlabel{ErdosSzekeresGenPos}
For all $k\geq3$ and $\ell\geq3$,
$$\ES(k,\ell)\leq(\ell-3)\binom{\ES(k)-1}{2}+\ES(k)\enspace.$$
\end{lemma}

\begin{proof}
It is well known 
\citep{Erdos86,Erdos88,Erdos89,Furedi-SJDM91,Brass03} and easily proved\footnote{Let $P$ be a set of points in the plane with at most $\ell-1$ points collinear and at most $k-1$ points in general position. Let $S\subseteq P$ be a maximal set of points in general position. Thus every point in $P-S$ is collinear with two points in $S$. The set $S$ determines $\binom{|S|}{2}$ lines, each with at most $\ell-3$ points in $P-S$. Thus $|P|\leq\binom{|S|}{2}(\ell-3)+|S|\leq\binom{k-1}{2}(\ell-3)+k-1$.
That is, if $|P|\geq\binom{k-1}{2}(\ell-3)+k$ then $P$ contains $\ell$ collinear points or $k$ points in general position.} that every set of at least $(\ell-3)\binom{k-1}{2}+k$ points in the plane contains $\ell$ collinear points or $k$ points in general position. 
Thus every set $P$ of at least $(\ell-3)\binom{\ES(k)-1}{2}+\ES(k)$ points in the plane contains $\ell$ collinear points or  $\ES(k)$ points in general position. In the latter case, $P$ contains $k$ points in strictly convex position. 
\end{proof}

The best known upper bound on $\ES(k)$, due to \citet{TothValtr05}, is
$$\ES(k)\leq\binom{2k-5}{k-2}+1\in O\left(\frac{2^{2k}}{\sqrt{k}}\right)\enspace.$$
Thus \lemref{ErdosSzekeresConvexToStrict} implies that if $k$ is odd then
\begin{equation}
\eqnlabel{ErdosSzekeresConvexToStrictOdd}
\ES(k,\ell)\in 
 O\left(\frac{2^{(k-1)(\ell-1)}}{\sqrt{k\ell}}\right)\enspace,
\end{equation}
and if $k$ is even then
\begin{equation}
\eqnlabel{ErdosSzekeresConvexToStrictEven}
\ES(k,\ell)\in 
  O\left(\frac{2^{(k-2)(\ell-1)}}{\sqrt{k\ell}}\right)\enspace.
\end{equation}
Similarly, \lemref{ErdosSzekeresGenPos} implies that 
\begin{equation}
\eqnlabel{ErdosSzekeresGenPos}
\ES(k,\ell)\in O\left(\frac{\ell\cdot 2^{4k}}{k}\right)\enspace.
\end{equation}
Note that the bound in \eqnref{ErdosSzekeresGenPos} is stronger than the bounds in \twoeqnref{ErdosSzekeresConvexToStrictOdd}{ErdosSzekeresConvexToStrictEven} for $\ell\geq6$ and sufficiently large $k$.  For $\ell\leq 5$ the 
bounds in \twoeqnref{ErdosSzekeresConvexToStrictOdd}{ErdosSzekeresConvexToStrictEven} are stronger.

\section{Empty Quadrilaterals}
\seclabel{Quads}

Point sets with no $4$-hole are characterised as follows.

\begin{theorem}[\citep{DESW-CGTA,Eppstein-SOCG07}]
\lemlabel{Eppstein}
The following are equivalent for a finite set of points $P$:
\begin{enumerate}[(a)]
\item $P$ contains no $4$-hole,
\item the visibility graph of $P$ is crossing-free,
\item $P$ has a unique triangulation,
\item at least one of the following conditions hold:
\begin{itemize}
\item all the points in $P$, except for at most one, are collinear; see \figsref{PlanarVisGraphs}(a) and (b),
\item there are two points $v,w\in P$ on opposite sides of some line $L$, such that $P-\{v,w\}\subseteq L$ and the intersection of $\conv(P-\{v,w\})$ and $\overline{vw}$ either is a point in $P-\{v,w\}$ or is empty; see \figsref{PlanarVisGraphs}(c) and (d),
\item $P$ is a set of six points with the same order type as the set illustrated in \figref{PlanarVisGraphs}(e).
\end{itemize}
\end{enumerate}
\end{theorem}


\Figure{PlanarVisGraphs}{\includegraphics{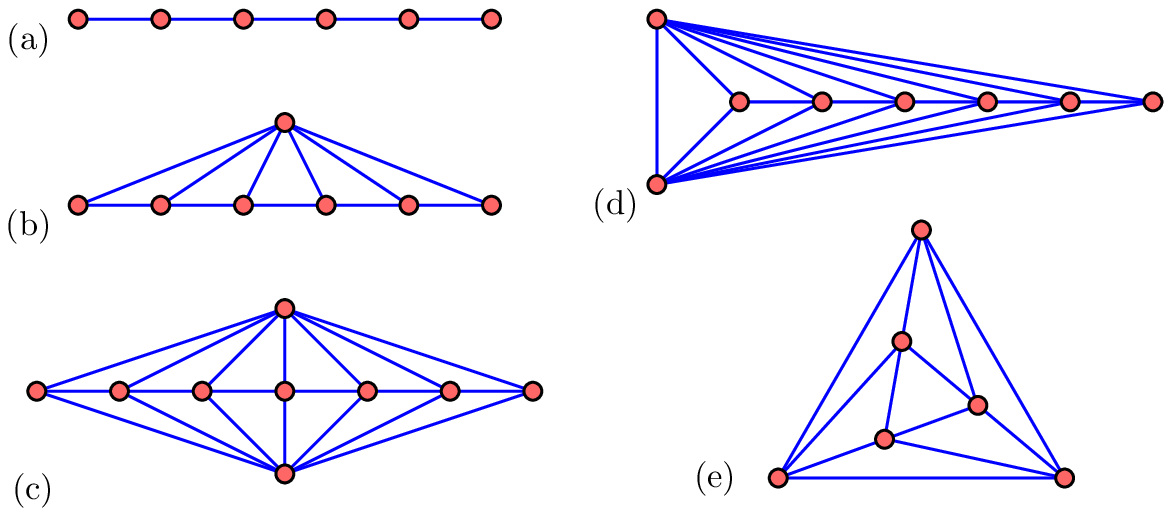}}{The point sets with no $4$-hole.} 

\begin{corollary}
\corlabel{Quadrilateral}
For every integer $\ell\geq2$, every set of at least $\max\{7,\ell+2\}$ points in the plane contains $\ell$ collinear points or a $4$-hole.
\end{corollary}

\corref{Quadrilateral} enables a third proof of \thmref{StrictErdosSzekeres}: 
By the $2$-colour Ramsey Theorem for hypergraphs (see \citep{RamseyTheory}), for every integer $t$ there is an integer $n$ such that  for every $2$-colouring of the edges of any complete $4$-uniform hypergraph on at least $n$ vertices, there is a set $X$ of $t$ vertices such that the edges induced by $X$ are monochromatic. 
Apply this result with $t:=\max\{7,k,\ell+2\}$. We claim that $\ES(k,\ell)\leq n$. 
Let $P$ be a set of at least $n$ points in the plane with no $\ell$ collinear points.
Let $G$ be the complete $4$-uniform hypergraph with vertex set $P$.
For each $4$-tuple $T$ of vertices, colour the edge $T$ \emph{blue} if $T$ forms a strictly convex quadrilateral, and \emph{red} otherwise.
Thus there is a set $X$ of $t$ points such that the edges induced by $X$ are monochromatic. 
If all the edges induced by $X$ are red, then no $4$-tuple of points in $X$ forms a strictly convex quadrilateral, which contradicts \corref{Quadrilateral} since $|X|\geq\max\{7,\ell+2\}$. 
Otherwise, all the edges induced by $X$ are blue.
That is, every $4$-tuple of vertices in $X$ forms a strictly convex quadrilateral.
This implies that $X$ forms a strictly convex $t$-gon (for otherwise some non-corner in $X$ would be in a triangle of corners of $X$, implying there is a $4$-tuple of points in $X$ that do not form a
strictly convex quadrilateral). Since $t\geq k$ we are done.

\section{Empty Pentagons}
\seclabel{Pentagons}

In this section, we prove our main result, \thmref{GenFiveHole}. The proof loosely follows the proof of the $6$-hole theorem for points in general position by \citet{Valtr-EmptyHexagon}, which in turn is a simplification of the proof by \citet{Gerken-DCG08}. 

\begin{proof}[Proof of \thmref{GenFiveHole}]
Fix $\ell\geq3$ and let $k:=\frac{(2\ell-1)^{\ell}-1}{2\ell-2}$, which is an integer.

Let $P$ be a set of at least $\ES(k)$ points in the plane.
By \thmref{ErdosSzekeres}, $P$ contains $k$ points in convex position.
Suppose for the sake of contradiction that $P$ contains no $\ell$ collinear points and no $5$-hole.

A set $X$ of at least $k$ points in $P$ in convex position is said to be  \emph{$k$-minimal} if there is no set $Y$ of at least $k$ points in $P$ in convex position, such that $\conv(Y)\subsetneq\conv(X)$.

As illustrated in \figref{Layers}, let $A_1$ be a $k$-minimal subset of $P$.
Let $A_2,\dots,A_{\ell-1}$ be the convex layers inside $A_1$.
More precisely, for $i=2,\dots,\ell-1$, let $A_i$ be the set of points in $P$ on the boundary of the convex hull of $(P\cap \conv(A_{i-1}))-A_{i-1}$.
Let $A_{\ell}:=(P\cap \conv(A_{\ell-1}))-A_{\ell-1}$.

\Figure{Layers}{\includegraphics{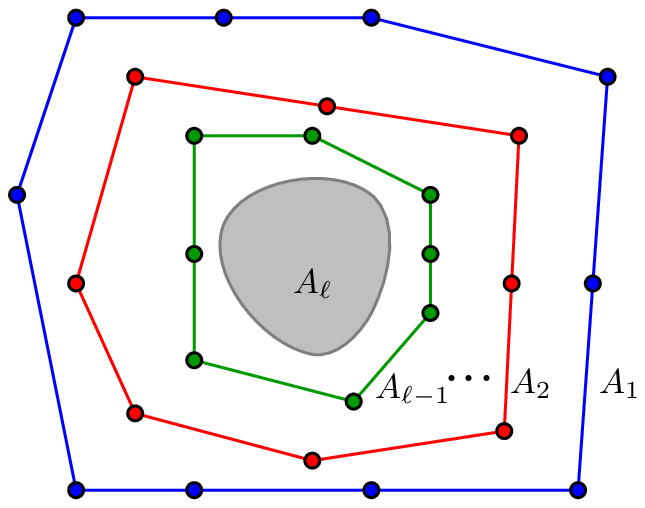}}{Definition of $A_1,\dots,A_{\ell}$.}

By \lemref{ConvexToStrict} with $k=5$, for each $i\in[2,\ell]$, 
any $2\ell-1$ consecutive points of $A_{i-1}$ contains five points in strictly convex position. 
Thus the convex hull of any $2\ell-1$ consecutive points of $A_{i-1}$ contains a point in $A_i$, as otherwise it would contain a $5$-hole. 
Now $A_{i-1}$ contains $\FLOOR{\frac{|A_{i-1}|}{2\ell-1}}$ disjoint subsets, each consisting of $2\ell-1$ consecutive points, and the convex hull of each subset contains a point in $A_i$. Since the convex hulls of these subsets of $A_{i-1}$ are disjoint,
$$|A_i|\geq\FLOOR{\frac{|A_{i-1}|}{2\ell-1}}>\frac{|A_{i-1}|}{2\ell-1}-1\enspace,$$
implying
\begin{equation}
\eqnlabel{ConsecutiveLayers}
|A_{i-1}|<(2\ell-1)(|A_i|+1)\enspace. 
\end{equation}

Suppose that $A_i=\varnothing$ for some $i\in[2,\ell]$. 
By \eqnref{ConsecutiveLayers}, $|A_{i-1}| <2\ell-1$ and $|A_{i-2}|<(2\ell-1)^2+(2\ell-1)$, and by induction,
$$|A_1|<\sum_{j=1}^{i-1}(2\ell-1)^j<\frac{(2\ell-1)^{i}-1}{2\ell-2}\leq\frac{(2\ell-1)^{\ell}-1}{2\ell-2}=k\enspace,$$ which is a contradiction.
Now assume that $A_i\neq\varnothing$ for all $i\in[\ell]$.
Fix a  point $z\in A_{\ell}$.

Note that if $|A_i|\leq2$ for some $i\in[\ell-1]$ then $A_{i+1}=\varnothing$.
Thus we may assume that $|A_i|\geq3$ for all $i\in[\ell-1]$.
Consider each such set $A_i$ to be ordered clockwise around $\conv(A_i)$. 
If $x$ and $y$ are consecutive points in $A_i$ with $y$ clockwise from $x$ then we say that the oriented segment $\overrightarrow{xy}$ is an \emph{arc} of $A_i$.

Let $\overrightarrow{xy}$ be an arc of $A_i$ for some $i\in[\ell-2]$.
We say that $\overrightarrow{xy}$ is \emph{empty} if $\Delta(x,y,z)\cap A_{i+1}=\varnothing$, as illustrated in \figref{Follower}(a).
In this case, the intersection of the boundary of $\conv(A_{i+1})$ and $\Delta(x,y,z)$ is contained in an arc $\overrightarrow{pq}$.
We call $\overrightarrow{pq}$ the \emph{follower} of $\overrightarrow{xy}$.

\Figure{Follower}{\includegraphics{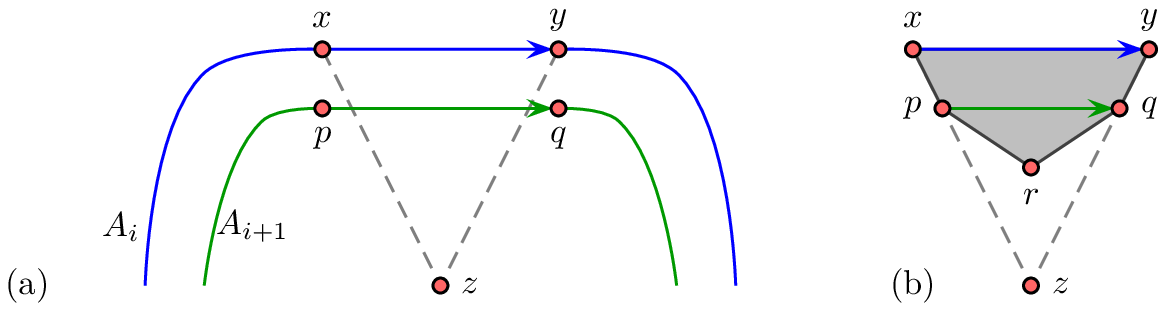}}{}

\begin{claim}
\claimlabel{B}
If $\overrightarrow{pq}$ is the follower of an empty arc $\overrightarrow{xy}$, 
then  $\{x,y,p,q\}$ is a $4$-hole and $\overrightarrow{pq}$ is empty.
\end{claim}

\begin{proof}
Say $\overrightarrow{xy}$ is an arc of $A_i$, where $i\in[\ell-2]$.
Let $S:=\{x,y,p,q\}$. 
Since $p$ and $q$ are in the interior of $\conv(A_i)$, both $x$ and $y$ are corners of $S$.
Both $p$ and $q$ are corners of $S$, as otherwise $\overrightarrow{xy}$ is not empty. 
Thus $S$ is in strictly convex position. $S$ is empty by the definition of $A_{i+1}$. 
Thus $S$ is a $4$-hole.

Suppose that $\overrightarrow{pq}$ is not empty; that is, $\Delta(p,q,z)\cap A_{i+2}\neq\varnothing$.
Let $r$ be a point in $\Delta(p,q,z)\cap A_{i+2} $ closest to $\overline{pq}$.
Thus $\Delta(p,q,r)\cap P=\varnothing$. Since $\{x,y,p,q\}$ is a $4$-hole, $\{x,y,p,q,r\}$ is a $5$-hole, as illustrated in \figref{Follower}(b). 
This contradiction proves that $\overrightarrow{pq}$ is empty.
\end{proof}

As illustrated in \figref{Aligned}(a)--(c), we say the follower $\overrightarrow{pq}$ of $\overrightarrow{xy}$ is:
\begin{itemize}
\item \emph{double-aligned} if $p\in\overline{xz}$ and $q\in\overline{yz}$, 
\item  \emph{left-aligned} if $p\in\overline{xz}$ and $q\not\in\overline{yz}$, 
\item  \emph{right-aligned} if $p\not\in\overline{xz}$ and $q\in\overline{yz}$.
\end{itemize}

\Figure{Aligned}{\includegraphics{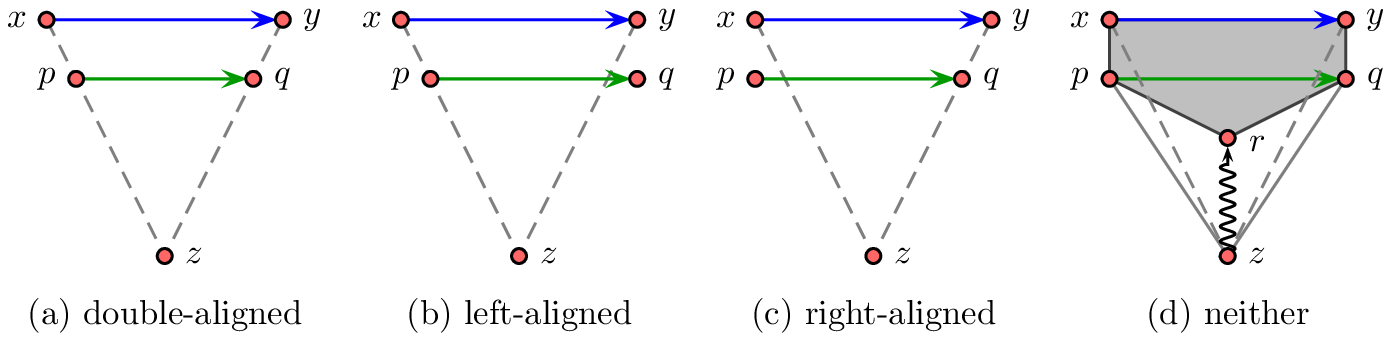}}{}

\begin{claim}
\claimlabel{C}
If $\overrightarrow{pq}$ is the follower of an empty arc $\overrightarrow{xy}$, 
then $\overrightarrow{pq}$ is either double-aligned or left-aligned or right-aligned.
\end{claim}

\begin{proof}
Suppose that $\overrightarrow{pq}$ is neither double-aligned nor left-aligned nor right-aligned, as illustrated in \figref{Aligned}(d).
Since $\overrightarrow{xy}$ is empty, $p\not\in \Delta[x,y,z]$ and $q\not\in\Delta[x,y,z]$.
Let $D:=(P\cap\Delta[p,q,z])-\{p,q\}$. Thus $z\in D$ and $D\neq\varnothing$.
Let $r$ be a point in $D$ closest to $\overline{pq}$. Thus $\Delta(r,p,q)$ is empty. 
By \claimref{B}, $\{x,y,p,q\}$ is a $4$-hole. Thus $\{x,y,p,q,r\}$ is a $5$-hole, which is the desired contradiction.
\end{proof}

Suppose that no arc of $A_1$ is empty.
That is, $\Delta(x,y,z)\cap A_2\neq\varnothing$ for each arc $\overrightarrow{xy}$ of $A_1$.
Observe that $\Delta(x,y,z)\cap \Delta(p,q,z)=\varnothing$ for distinct arcs $\overrightarrow{xy}$ and $\overrightarrow{pq}$ of $A_1$ 
(since these triangles are open). 
Thus $|A_2|\geq|A_1|$, which contradicts the minimality of $A_1$.

Now assume that some arc $\overrightarrow{x_1y_1}$ of $A_1$ is empty.
For $i=2,3,\dots,\ell-1$, let $\overrightarrow{x_iy_i}$ be the follower of $\overrightarrow{x_{i-1}y_{i-1}}$.
By \claimref{B} (at each iteration), $\overrightarrow{x_iy_i}$ is empty.
For some $i\in[2,\ell-2]$, the arc $\overrightarrow{x_iy_i}$ is not double-aligned, 
as otherwise $\{x_1,x_2,\dots,x_{\ell-2},z\}$ are collinear and $\{y_1,y_2,\dots,y_{\ell-2},z\}$ are collinear, 
which implies that 
$\{x_1,x_2,\dots,x_{\ell-1},z\}$ are collinear or $\{y_1,y_2,\dots,y_{\ell-1},z\}$ are collinear by \claimref{C}.
Let $i$ be the minimum integer in  $[2,\ell-2]$ such that $\overrightarrow{x_iy_i}$ is not double-aligned.
Without loss of generality, $\overrightarrow{x_iy_i}$ is left-aligned. 
On the other hand, $\overrightarrow{x_jy_j}$ is not left-aligned for all $j\in[i+1,\ell-1]$, as otherwise $\{x_1,x_2,\dots,x_{\ell-1},z\}$ are collinear.
Let $j$ be the minimum integer in $[i+1,\ell-1]$ such that $\overrightarrow{x_jy_j}$ is not left-aligned.
Thus   $\overrightarrow{x_{j-1}y_{j-1}}$ is left-aligned and  $\overrightarrow{x_jy_j}$ is not left-aligned.
It follows that $\{x_{j-2},y_{j-2},y_{j-1},y_{j},x_{j-1}\}$ is a $5$-hole, as illustrated in \figref{Hole}.
This contradiction proves that $P$ contains $\ell$ collinear points or a $5$-hole.
\end{proof}



\Figure{Hole}{\includegraphics{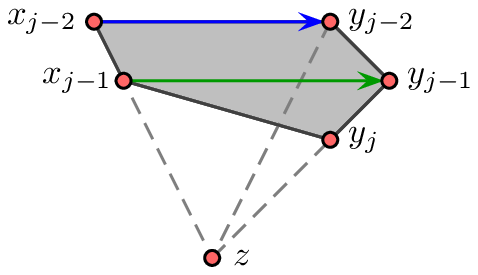}}{}

We expect that the lower bound on $|P|$ in \thmref{GenFiveHole} is far from optimal.
All known point sets with at most $\ell$ collinear points and no $5$-hole have $O(\ell^2)$ points, the $\ell\times\ell$ grid for example. 
See \citep{KPW-DCG05,Eppstein-SOCG07} for other examples.

\medskip\noindent
\textbf{Open Problem.} 
For which values of $\ell$  is there an integer $n$ such that every set of at least $n$ points in the plane contains $\ell$ collinear points or a $6$-hole?

\medskip

This is true for $\ell=3$ by the empty hexagon theorem. If this question is true for a particular value of $\ell$ then \conjref{BigClique} is true for $k=6$ and the same value of $\ell$. For $k\geq7$ different methods are needed since there are point sets in general position with no $7$-hole.

\section*{Acknowledgements}

This research was initiated at The 24th Bellairs Winter Workshop on Computational Geometry, held in February 2009 at the Bellairs Research Institute of McGill University in Barbados. The authors are grateful to Godfried Toussaint and Erik Demaine for organising the workshop, and to the other workshop participants for providing a stimulating working environment. 


\begin{thebibliography}{41}
\providecommand{\natexlab}[1]{#1}
\providecommand{\url}[1]{\texttt{#1}}
\providecommand{\urlprefix}{}
\expandafter\ifx\csname urlstyle\endcsname\relax
  \providecommand{\doi}[1]{doi:\discretionary{}{}{}#1}\else
  \providecommand{\doi}{doi:\discretionary{}{}{}\begingroup
  \urlstyle{rm}\Url}\fi

\bibitem[{Addario-Berry et~al.(2007)Addario-Berry, Fernandes, Kohayakawa,
  de~Pina, and Wakabayashi}]{Luigi}
\textsc{Louigi Addario-Berry, Cristina Fernandes, Yoshiharu Kohayakawa,
  Jos~Coelho de~Pina, and Yoshiko Wakabayashi}.
\newblock On a geometric {R}amsey-style problem, 2007.
\newblock \urlprefix\url{http://crm.umontreal.ca/cal/en/mois200708.html}.

\bibitem[{B{\'a}r{\'a}ny and K{\'a}rolyi(2001)}]{BK-LNCS01}
\textsc{Imre B{\'a}r{\'a}ny and Gyula K{\'a}rolyi}.
\newblock Problems and results around the {E}rd{\H o}s-{S}zekeres convex
  polygon theorem.
\newblock In \emph{Proc. Japanese Conf. on Discrete and Computational Geometry
  (JCDCG 2000)}, vol. 2098 of \emph{Lecture Notes in Comput. Sci.}, pp.
  91--105. Springer, 2001.

\bibitem[{B{\'a}r{\'a}ny and Valtr(2004)}]{BaranyValtr-SSMH04}
\textsc{Imre B{\'a}r{\'a}ny and Pavel Valtr}.
\newblock Planar point sets with a small number of empty convex polygons.
\newblock \emph{Studia Sci. Math. Hungar.}, 41(2):243--266, 2004.

\bibitem[{Bisztriczky et~al.(2004)Bisztriczky, Hosono, K{\'a}rolyi, and
  Urabe}]{BHKU-PMH04}
\textsc{Tibor Bisztriczky, Kiyoshi Hosono, Gyula K{\'a}rolyi, and Masatsugu
  Urabe}.
\newblock Constructions from empty polygons.
\newblock \emph{Period. Math. Hungar.}, 49(2):1--8, 2004.

\bibitem[{Bra{\ss}(2003)}]{Brass03}
\textsc{Peter Bra{\ss}}.
\newblock On point sets without {$k$} collinear points.
\newblock In \emph{Discrete geometry}, vol. 253 of \emph{Monogr. Textbooks Pure
  Appl. Math.}, pp. 185--192. Dekker, 2003.

\bibitem[{Bra{\ss} et~al.(2005)Bra{\ss}, Moser, and Pach}]{BMP}
\textsc{Peter Bra{\ss}, William O.~J. Moser, and J{\'a}nos Pach}.
\newblock \emph{Research problems in discrete geometry}.
\newblock Springer, 2005.

\bibitem[{Dobkin et~al.(1990)Dobkin, Edelsbrunner, and Overmars}]{DEO-Algo90}
\textsc{David~P. Dobkin, Herbert Edelsbrunner, and Mark~H. Overmars}.
\newblock Searching for empty convex polygons.
\newblock \emph{Algorithmica}, 5(4):561--571, 1990.

\bibitem[{Du and Ding(2005)}]{DuDing-JAMC05}
\textsc{Yatao Du and Ren Ding}.
\newblock New proofs about the number of empty convex 4-gons and 5-gons in a
  planar point set.
\newblock \emph{J. Appl. Math. Comput.}, 19(1-2):93--104, 2005.

\bibitem[{Dujmovi{\'c} et~al.(2007)Dujmovi{\'c}, Eppstein, Suderman, and
  Wood}]{DESW-CGTA}
\textsc{Vida Dujmovi{\'c}, David Eppstein, Matthew Suderman, and David~R.
  Wood}.
\newblock Drawings of planar graphs with few slopes and segments.
\newblock \emph{Comput. Geom. Theory Appl.}, 38:194--212, 2007.

\bibitem[{Dumitrescu(2000)}]{Dumitrescu-SSMH00}
\textsc{Adrian Dumitrescu}.
\newblock Planar sets with few empty convex polygons.
\newblock \emph{Studia Sci. Math. Hungar.}, 36(1-2):93--109, 2000.

\bibitem[{Eppstein(2007)}]{Eppstein-SOCG07}
\textsc{David Eppstein}.
\newblock Happy endings for flip graphs.
\newblock In \emph{Proc. 23rd Annual Symposium on Computational Geometry (SoCG
  '07)}, pp. 92--101. ACM, 2007.

\bibitem[{Erd{\H{o}}s(1975)}]{Erdos75}
\textsc{Paul Erd{\H{o}}s}.
\newblock On some problems of elementary and combinatorial geometry.
\newblock \emph{Ann. Mat. Pura Appl. (4)}, 103:99--108, 1975.

\bibitem[{Erd{\H{o}}s(1986)}]{Erdos86}
\textsc{Paul Erd{\H{o}}s}.
\newblock On some metric and combinatorial geometric problems.
\newblock \emph{Discrete Math.}, 60:147--153, 1986.

\bibitem[{Erd{\H{o}}s(1988)}]{Erdos88}
\textsc{Paul Erd{\H{o}}s}.
\newblock Some old and new problems in combinatorial geometry.
\newblock In \emph{Applications of discrete mathematics}, pp. 32--37. SIAM,
  1988.

\bibitem[{Erd{\H{o}}s(1989)}]{Erdos89}
\textsc{Paul Erd{\H{o}}s}.
\newblock Problems and results on extremal problems in number theory, geometry,
  and combinatorics.
\newblock \emph{Rostock. Math. Kolloq.}, 38:6--14, 1989.

\bibitem[{Erd{\H{o}}s and Szekeres(1935)}]{ES35}
\textsc{Paul Erd{\H{o}}s and George Szekeres}.
\newblock A combinatorial problem in geometry.
\newblock \emph{Composito Math.}, 2:464--470, 1935.

\bibitem[{F{\"u}redi(1991)}]{Furedi-SJDM91}
\textsc{Zolt{\'a}n F{\"u}redi}.
\newblock Maximal independent subsets in {S}teiner systems and in planar sets.
\newblock \emph{SIAM J. Discrete Math.}, 4(2):196--199, 1991.

\bibitem[{Gerken(2008)}]{Gerken-DCG08}
\textsc{Tobias Gerken}.
\newblock Empty convex hexagons in planar point sets.
\newblock \emph{Discrete Comput. Geom.}, 39(1-3):239--272, 2008.

\bibitem[{Graham et~al.(1980)Graham, Rothschild, and Spencer}]{RamseyTheory}
\textsc{Ronald~L. Graham, Bruce~L. Rothschild, and Joel~H. Spencer}.
\newblock \emph{Ramsey theory}.
\newblock John Wiley, 1980.

\bibitem[{Harborth(1978)}]{Harborth78}
\textsc{Heiko Harborth}.
\newblock Konvexe {F}\"unfecke in ebenen {P}unktmengen.
\newblock \emph{Elem. Math.}, 33(5):116--118, 1978.

\bibitem[{Horton(1983)}]{Horton83}
\textsc{Joseph~D. Horton}.
\newblock Sets with no empty convex {$7$}-gons.
\newblock \emph{Canad. Math. Bull.}, 26(4):482--484, 1983.

\bibitem[{Hosono(2005)}]{Hosono-DM05}
\textsc{Kiyoshi Hosono}.
\newblock On the existence of a convex point subset containing one triangle in
  the plane.
\newblock \emph{Discrete Math.}, 305(1-3):201--218, 2005.

\bibitem[{K{\'a}ra et~al.(2005)K{\'a}ra, P{\'o}r, and Wood}]{KPW-DCG05}
\textsc{Jan K{\'a}ra, Attila P{\'o}r, and David~R. Wood}.
\newblock On the chromatic number of the visibility graph of a set of points in
  the plane.
\newblock \emph{Discrete Comput. Geom.}, 34(3):497--506, 2005.

\bibitem[{Koshelev(2007)}]{Koshelev-DAN07}
\textsc{V.~A. Koshelev}.
\newblock The {E}rd{\H o}s-{S}zekeres problem.
\newblock \emph{Dokl. Akad. Nauk}, 415(6):734--736, 2007.

\bibitem[{Kun and Lippner(2003)}]{KL-PMH03}
\textsc{G{\'a}bor Kun and G{\'a}bor Lippner}.
\newblock Large empty convex polygons in {$k$}-convex sets.
\newblock \emph{Period. Math. Hungar.}, 46(1):81--88, 2003.

\bibitem[{Matou{\v{s}}ek(2002)}]{Mat02}
\textsc{Ji{\v{r}}{\'\i} Matou{\v{s}}ek}.
\newblock \emph{Lectures on Discrete Geometry}, vol. 212 of \emph{Graduate
  Texts in Mathematics}.
\newblock Springer, 2002.
\newblock ISBN 0-387-95373-6.

\bibitem[{Matou{\v{s}}ek(2008)}]{Matousek08}
\textsc{Ji{\v{r}}{\'i} Matou{\v{s}}ek}.
\newblock Blocking visibility for points in general position, 2008.
\newblock \urlprefix\url{http://kam.mff.cuni.cz/~matousek/gblock.pdf}.
\newblock Submitted.

\bibitem[{Morris and Soltan(2000)}]{MorrisSoltan-BAMS00}
\textsc{Walter~D. Morris and Valeriu Soltan}.
\newblock The {E}rd{\H o}s-{S}zekeres problem on points in convex position---a
  survey.
\newblock \emph{Bull. Amer. Math. Soc. (N.S.)}, 37(4):437--458, 2000.

\bibitem[{Nicol{\'a}s(2007)}]{Nicolas-DCG07}
\textsc{Carlos~M. Nicol{\'a}s}.
\newblock The empty hexagon theorem.
\newblock \emph{Discrete Comput. Geom.}, 38(2):389--397, 2007.

\bibitem[{Nyklov{\'a}(2003)}]{Nyklova-SSMH03}
\textsc{Helena Nyklov{\'a}}.
\newblock Almost empty polygons.
\newblock \emph{Studia Sci. Math. Hungar.}, 40(3):269--286, 2003.

\bibitem[{Overmars(2003)}]{Overmars-DCG03}
\textsc{Mark Overmars}.
\newblock Finding sets of points without empty convex 6-gons.
\newblock \emph{Discrete Comput. Geom.}, 29(1):153--158, 2003.

\bibitem[{Pinchasi et~al.(2006)Pinchasi, Radoi\v{c}i\'{c}, and
  Sharir}]{PRS-JCTA06}
\textsc{Rom Pinchasi, Rado\v{s} Radoi\v{c}i\'{c}, and Micha Sharir}.
\newblock On empty convex polygons in a planar point set.
\newblock \emph{J. Combin. Theory Ser. A}, 113(3):385--419, 2006.

\bibitem[{Rabinowitz(2005)}]{Rabinowitz}
\textsc{Stanley Rabinowitz}.
\newblock Consequences of the pentagon property.
\newblock \emph{Geombinatorics}, 14:208--220, 2005.

\bibitem[{T{\'o}th and Valtr(2005)}]{TothValtr05}
\textsc{G{\'e}za T{\'o}th and Pavel Valtr}.
\newblock The {E}rd{\H o}s-{S}zekeres theorem: upper bounds and related
  results.
\newblock In \emph{Combinatorial and computational geometry}, vol.~52 of
  \emph{Math. Sci. Res. Inst. Publ.}, pp. 557--568. Cambridge Univ. Press,
  2005.

\bibitem[{Valtr(1992)}]{Valtr-DCG92}
\textsc{Pavel Valtr}.
\newblock Convex independent sets and {$7$}-holes in restricted planar point
  sets.
\newblock \emph{Discrete Comput. Geom.}, 7(2):135--152, 1992.

\bibitem[{Valtr(1995)}]{Valtr-SSMH95}
\textsc{Pavel Valtr}.
\newblock On the minimum number of empty polygons in planar point sets.
\newblock \emph{Studia Sci. Math. Hungar.}, 30(1-2):155--163, 1995.

\bibitem[{Valtr(2002)}]{Valtr-DCG02}
\textsc{Pavel Valtr}.
\newblock A sufficient condition for the existence of large empty convex
  polygons.
\newblock \emph{Discrete Comput. Geom.}, 28(4):671--682, 2002.

\bibitem[{Valtr(2008)}]{Valtr-EmptyHexagon}
\textsc{Pavel Valtr}.
\newblock On empty hexagons.
\newblock In \emph{Surveys on discrete and computational geometry}, vol. 453 of
  \emph{Contemp. Math.}, pp. 433--441. Amer. Math. Soc., 2008.

\bibitem[{Valtr et~al.(2007)Valtr, Lippner, and K{\'a}rolyi}]{VLK-PMH07}
\textsc{Pavel Valtr, G{\'a}bor Lippner, and Gyula K{\'a}rolyi}.
\newblock Empty convex polygons in almost convex sets.
\newblock \emph{Period. Math. Hungar.}, 55(2):121--127, 2007.

\bibitem[{Wu and Ding(2007)}]{WuDing07}
\textsc{Liping Wu and Ren Ding}.
\newblock Reconfirmation of two results on disjoint empty convex polygons.
\newblock In \emph{Discrete geometry, combinatorics and graph theory}, vol.
  4381 of \emph{Lecture Notes in Comput. Sci.}, pp. 216--220. Springer, 2007.

\bibitem[{Wu and Ding(2008)}]{WuDing-AML08}
\textsc{Liping Wu and Ren Ding}.
\newblock On the number of empty convex quadrilaterals of a finite set in the
  plane.
\newblock \emph{Appl. Math. Lett.}, 21(9):966--973, 2008.

\end{thebibliography}

\def\cprime{$'$} \def\soft#1{\leavevmode\setbox0=\hbox{h}\dimen7=\ht0\advance
  \dimen7 by-1ex\relax\if t#1\relax\rlap{\raise.6\dimen7
  \hbox{\kern.3ex\char'47}}#1\relax\else\if T#1\relax
  \rlap{\raise.5\dimen7\hbox{\kern1.3ex\char'47}}#1\relax \else\if
  d#1\relax\rlap{\raise.5\dimen7\hbox{\kern.9ex \char'47}}#1\relax\else\if
  D#1\relax\rlap{\raise.5\dimen7 \hbox{\kern1.4ex\char'47}}#1\relax\else\if
  l#1\relax \rlap{\raise.5\dimen7\hbox{\kern.4ex\char'47}}#1\relax \else\if
  L#1\relax\rlap{\raise.5\dimen7\hbox{\kern.7ex
  \char'47}}#1\relax\else\message{accent \string\soft \space #1 not
  defined!}#1\relax\fi\fi\fi\fi\fi\fi} \def\Dbar{\leavevmode\lower.6ex\hbox to
  0pt{\hskip-.23ex \accent"16\hss}D}

\end{document}